\documentclass[11pt]{article}
\usepackage[T1]{fontenc}
\usepackage[utf8]{inputenc}
\usepackage{lmodern}
\usepackage[a4paper,margin=27mm]{geometry}
\usepackage{amsmath,amssymb,amsthm}
\usepackage[hidelinks]{hyperref}
\usepackage{color}

\allowdisplaybreaks[1]
\newcommand{\R}{\mathbb R}
\newcommand{\bbern}{\mathbf b}
\theoremstyle{plain}
\newtheorem{theorem}{Theorem}[section]
\newtheorem{lemma}[theorem]{Lemma}
\newtheorem{proposition}[theorem]{Proposition}

\theoremstyle{remark}

\title{The Berry--Esseen Estimate\\
in the Boolean Central Limit Theorem}
\author{Makoto Maejima \and Noriyoshi Sakuma}
\date{\today}

\begin{document}
\maketitle

\begin{abstract}
We estimate the L\'evy distance to the symmetric Bernoulli distribution
for normalized Boolean convolution powers of a probability measure with mean zero and variance one.
If the absolute moment of order $r$ is finite, we obtain the upper bounds
$o(n^{-(r-2)/6})$ for $2<r<5$ and $O(n^{-1/2})$ for $r\ge5$.
The proof relies on an inequality that controls the L\'evy distance to the Bernoulli law by the values of the reciprocal Cauchy transform at two points.
For each $r\ge2$, we also construct a single fixed distribution whose absolute moments
are finite exactly up to order $r$, and show that the exponent $\min\{(r-2)/6,1/2\}$ cannot be improved.
Under only a finite second moment assumption, the convergence can be arbitrarily slow.
\end{abstract}


\section{Introduction}

In this paper, we study the rate of convergence in the Boolean central limit theorem.
In the classical and free central limit theorem, Berry--Esseen estimates are formulated for the Kolmogorov distance, which is the supremum of the absolute distance of the distribution functions of two probability measures.
However, in the Boolean case, the normalized partial sums of Boolean independent random variables do not necessarily converge in the sense of Kolmogorov distance,
since the limiting distribution (the symmetric Bernoulli distribution) 
$$\bbern=\frac12(\delta_{-1}+\delta_1)$$ has atoms.
For example, we can see this fact by Example $1$ of \cite{AS18}, because Kolmogorov distance of this example and $\bbern$ is always $\frac{1}{2}$.
Therefore we use a weaker distance, which is the L\'evy distance defined below, 
whose convergence to $0$ is equivalent to weak convergence of measures.
For probability measures $\mu$ and $\nu$ on the real line, the L\'evy distance of $\mu$ and $\nu$ is defined by
\[
 L(\mu,\nu)=\inf\{\varepsilon>0:
 \mu((-\infty, x-\varepsilon])-\varepsilon\le\nu((-\infty,x])
 \le \mu((-\infty,x+\varepsilon])+\varepsilon\quad(x\in\R)\}.
\]

For a probability measure $\mu$, we write
\[
 m_j(\mu)=\int_\R x^j\,\mu(dx),\qquad
 \beta_r(\mu)=\int_\R |x|^r\,\mu(dx)\quad(r>0).
\]
We use $m_j$ only for integer orders for which the corresponding absolute moment is finite.
Throughout this paper, we assume that $m_1(\mu)=0$ and $m_2(\mu)=1$, and put
\[
 \mu_n=D_{1/\sqrt n}\mu^{\uplus n}.
\]
Here $\uplus$ denotes Boolean convolution, and $D_c\mu$ is the image measure
of $\mu$ under the map $x\mapsto cx$.
The Boolean central limit theorem \cite[Theorem~3.4]{SW97} states that $\mu_n$ converges weakly to $\bbern$.

The purpose of this paper is to give upper bounds for the L\'evy distance $L(\mu_n,\bbern)$ in terms of the order $r$ of a finite absolute moment of $\mu$.
We also show that the resulting exponent is optimal, by constructing fixed distributions $\mu$ for which it cannot be improved.
Our main theorem, Theorem~\ref{thm:r-upper}, states that there exists an absolute constant $C$ such that
\begin{equation}\label{eq:overview}
 L(\mu_n,\bbern)\le
 \begin{cases}
 C\beta_r(\mu)^{1/3}n^{-(r-2)/6},&2<r\le5,\\
 C\beta_r(\mu)^{1/(r-2)}n^{-1/2},&r\ge5.
 \end{cases}
\end{equation}
In Theorem~\ref{thm:family} and Proposition~\ref{prop:drift}, we show that the exponent in \eqref{eq:overview} is best possible: for each $r>2$ there is a fixed $\mu$ with $\beta_r(\mu)<\infty$ for which $L(\mu_n,\bbern)=O(n^{-\gamma})$ fails for every $\gamma>\min\{(r-2)/6,1/2\}$.

Furthermore, under the second moment assumption alone no rate of convergence can be guaranteed: Theorem~\ref{thm:arbitrary-slow} shows that for every positive sequence $\varepsilon_n\to0$ there exists a single $\mu$ with $\limsup_{n}L(\mu_n,\bbern)/\varepsilon_n=\infty$.

Let us describe previous research on this problem.
Arizmendi and Salazar \cite{AS18} proved
\begin{align}\label{1}
L(\mu_n,\bbern)\le\frac72\bigl((m_4(\mu)-1)/n\bigr)^{1/3}
\end{align}
 under a finite fourth moment assumption \cite[Theorem~1]{AS18}  and $L(\mu_n,\bbern)\le2K/ n^{1/2}$ 
if $\operatorname{supp}\mu\subset[-K,K]$ \cite[Theorem~2]{AS18}.
To get \eqref{1} they first proved  
\begin{align}\label{2}
L(\mu,\mathbf b)\le \frac72\left(m_4(\mu)-1\right)^{1/3}
\end{align}
\cite[Theorem 3]{AS18} and also the optimality of the exponent $1/3$ in \eqref{2} uniformly over $\mu$ \cite[Proposition~3]{AS18}.
However, this does not imply the optimality of the convergence rate for the Boolean convolution powers of a fixed probability measure.
In contrast, Theorem~\ref{thm:family} with \(r=4\) provides a single probability measure
 \(\mu\) with \(m_4(\mu)<\infty\) such that \(L(\mu_n,\mathbf b)\) is not \(O(n^{-\gamma})\) for any \(\gamma>1/3\).

Our study was motivated by Salazar \cite{Salazar22}, who obtained the rate $n^{-1/2}$ under a finite sixth moment assumption, with a constant depending on $\mu$ and for all sufficiently large $n$. 
It is natural to ask whether a finite fourth moment, or even a finite absolute third moment, suffices for the rate $n^{-1/2}$. 
Our results answer this question in the negative: by Theorem~\ref{thm:family}, no absolute moment condition of order $r<5$ implies $L(\mu_n,\bbern)=O(n^{-1/2})$, whereas by Theorem~\ref{thm:r-upper} the absolute moment of order five does, with an absolute constant and for all $n\ge1$. 
In particular, \eqref{eq:overview} with $r=5$ improves the result of \cite{Salazar22}.

 Berry--Esseen type bounds for sums of Boolean independent random variables were also obtained in the operator-valued setting by Arizmendi, Banna and Tseng \cite{ABT} by a Lindeberg method, in terms of second and fourth moments. 
Hamdan \cite{Hamdan26} obtained the rate $n^{-1/2}$ under a finite third moment assumption, but for a distance defined through the self-energy $z-F_\mu(z)$ on a Stolz angle; his estimate implies weak convergence but does not yield a rate for the L\'evy distance. 
For comparison, in the classical and in the free central limit theorem a finite absolute third moment gives the rate $n^{-1/2}$ in the Kolmogorov distance (see, for example, \cite{Petrov} and \cite{MS}), whereas in the Boolean central limit theorem a finite absolute third moment gives only the exponent $1/6$ (Theorem~\ref{thm:r-upper} with $r=3$), and this exponent is sharp (Theorem~\ref{thm:family} with $r=3$).

The organization of this paper is the following.
In Section~\ref{sec:two}, we prove a two-point estimate using windows.
In Section~\ref{sec:upper}, we derive upper bounds from moment assumptions.
Section~\ref{sec:sharpness} proves optimality using fixed distributions,
and Section~\ref{sec:variance} treats the case of only a finite second moment.
The letter $C$ without a subscript denotes an absolute constant which may change from line to line.
In $O_{\mu,r}$, dependence on the indicated parameters is allowed.
The notation $\asymp$ means comparison from above and below by positive constant multiples;
subscripts indicate the dependence of these constants.


\section{A two-point estimate using windows}\label{sec:two}
For a probability measure $\mu$, let
\[
 G_\mu(z)=\int_\R\frac{\mu(dx)}{z-x},\qquad
 F_\mu(z)=\frac1{G_\mu(z)}\qquad(\operatorname{Im}z>0)
\]
be its Cauchy transform and reciprocal Cauchy transform, respectively.
We recall the standard dilation and Boolean convolution formulas
\cite[Sections~2.2--2.3]{Salazar22}: for $c>0$,
\[
 G_{D_c\mu}(z)=c^{-1}G_\mu(z/c),\qquad
 F_{D_c\mu}(z)=cF_\mu(z/c),\qquad
 F_{\mu^{\uplus n}}(z)=nF_\mu(z)-(n-1)z.
\]

\begin{lemma}[Two-point estimate]\label{lem:two}
For every probability measure $\mu$, $a\in\R$ and $0<y\le1/2$, we have
\begin{equation}\label{eq:two}
 L(\mu,\bbern)\le |a|+3y+
 \frac5y\max_{\sigma=\pm1}|F_\mu(a+\sigma+iy)-2iy|.
\end{equation}
\end{lemma}
\begin{proof}
We first suppose that $a=0$.

\smallskip\noindent
\emph{Step 1: From window masses to the L\'evy distance.}
Let $U_\sigma=[\sigma-y,\sigma+y]$ for $\sigma=\pm1$.
These windows are disjoint because $y\le1/2$.
For every $\delta\ge0$, we claim that
\begin{equation}\label{eq:mass-to-levy}
 \mu(U_\sigma)\ge\frac12-\delta\quad(\sigma=\pm1)
 \quad\Longrightarrow\quad L(\mu,\bbern)\le\max\{y,2\delta\}.
\end{equation}
If $0\le\delta<1/2$, the assumed lower bounds on the window masses are positive.
For each sign, define the measure
\[
 \mu_\sigma=
 \frac{1/2-\delta}{\mu(U_\sigma)}\,\mu|_{U_\sigma},
\]
where $\mu|_{U_\sigma}$ is the restriction of measure $\mu$ to $U_\sigma$.
The factor $(1/2-\delta)/\mu(U_\sigma)$ is at most one. 
Since the windows are disjoint,
$\mu_0:=\mu-\mu_{-1}-\mu_1$ is a positive measure of mass $2\delta$.
The measure on $\R^2$ given by
\[
 \pi=\mu_{-1}\otimes\delta_{-1}
       +\mu_1\otimes\delta_1
       +\frac12\mu_0\otimes(\delta_{-1}+\delta_1)
\]
has first marginal $\mu$ and second marginal $\bbern$.
For the first two terms, the distance between the two coordinates is at most $y$.
Thus, if $(X,Y)$ has law $\pi$, then
\[
 \mathbb P(|X-Y|>y)\le\mu_0(\R)=2\delta.
\]
On the event $\{|X-Y|\le y\}$ we have $\{X\le x-y\}\subset\{Y\le x\}\subset\{X\le x+y\}$.
Consequently,
\[
\mu((-\infty, x-y])-2\delta
\le \bbern((-\infty, x])
\le \mu((-\infty,x+y])+2\delta\qquad(x\in\R).
\]
Taking $\varepsilon=\max\{y,2\delta\}$ and using monotonicity of $\mu$
proves \eqref{eq:mass-to-levy} in this case.
If $\delta\ge1/2$, the same conclusion follows from $L\le1\le2\delta$.

\smallskip\noindent
\emph{Step 2: From the reciprocal transform to weighted masses.}
$\operatorname{Im}F_\mu(z)=-\operatorname{Im}G_\mu(z)/|G_\mu(z)|^2\ge\operatorname{Im}z$, because $-\operatorname{Im}G_\mu(z)=\operatorname{Im}z\int|z-x|^{-2}\,\mu(dx)\ge\operatorname{Im}z\,|G_\mu(z)|^2$ by the Cauchy--Schwarz inequality.
Put
\[
 E=\max_{\sigma=\pm1}|F_\mu(\sigma+iy)-2iy|,
 \qquad K_\sigma(x)=\frac{y^2}{(x-\sigma)^2+y^2}.
\]
For $z=\sigma+iy$, taking imaginary parts in the definition of $G_\mu$ gives
\[
 \int K_\sigma\,d\mu=-y\operatorname{Im}G_\mu(z),
 \qquad -y\operatorname{Im}\frac1{2iy}=\frac12.
\]
Also,
\[
 \frac1{F_\mu(z)}-\frac1{2iy}
 =\frac{2iy-F_\mu(z)}{2iyF_\mu(z)}.
\]
Thus, using $|\operatorname{Im}w|\le|w|$ and then $|F_\mu(z)|\ge y$, we obtain
\begin{equation}\label{eq:inverse-to-mass}
 \begin{aligned}
 \left|\int K_\sigma\, \mu(dx)-\frac12\right|
 &=y\left|\operatorname{Im}
       \left(\frac1{F_\mu(z)}-\frac1{2iy}\right)\right|\\
 &\le y\frac{|F_\mu(z)-2iy|}{2y|F_\mu(z)|}
 \le\frac{E}{2y}.
 \end{aligned}
\end{equation}
In particular, each of the two weighted masses is at least $1/2-E/(2y)$.

\smallskip\noindent
\emph{Step 3: From weighted masses to window masses.}
For every $x\in\R$, at least one of $|x-1|$ and $|x+1|$ is at least one.
Indeed, their sum is at least $|1-(-1)|=2$.
The corresponding kernel is at most $y^2/(1+y^2)\le y^2$,
and the other kernel is at most one. Hence
\[
 K_-+K_+\le1+y^2.
\]
Outside $U_-\cup U_+$, both distances exceed $y$, so both kernels are at most $1/2$.
Combining this with the bound $y^2$ for at least one kernel gives
\[
 K_-+K_+\le\frac12+y^2\quad\text{on }(U_-\cup U_+)^c.
\]
Let $q=\mu((U_-\cup U_+)^c)$. By summing the two lower bounds in
\eqref{eq:inverse-to-mass} and then splitting the integral into the windows and their complement,
\begin{align*}
 1-\frac Ey
 &\le\int(K_-+K_+)\,\mu(dx)\\
 &\le(1-q)(1+y^2)+q(1/2+y^2)
 =1+y^2-\frac q2.
\end{align*}
Rearranging yields $q\le2E/y+2y^2$.

Fix one sign $\sigma$. On the opposite window $U_{-\sigma}$,
\[
 |x-\sigma|\ge|\sigma-(-\sigma)|-|x+\sigma|
 \ge2-y\ge1,
\]
so $K_\sigma(x)\le y^2$. On its own window and on the complement of both windows,
we use $K_\sigma\le1$. Therefore
\begin{align*}
 \frac12-\frac{E}{2y}
 &\le\int K_\sigma\,d\mu\\
 &\le\mu(U_\sigma)+y^2\mu(U_{-\sigma})+q
 \le\mu(U_\sigma)+y^2+q.
\end{align*}
Substituting the bound for $q$, we get
\[
 \mu(U_\sigma)
 \ge\frac12-\frac{E}{2y}-y^2-\left(\frac{2E}{y}+2y^2\right)
 =\frac12-\frac{5E}{2y}-3y^2.
\]
Apply \eqref{eq:mass-to-levy} with $\delta=5E/(2y)+3y^2$.
Since $y\le1/2$, we have $6y^2\le3y$, and hence
\[
 L(\mu,\bbern)\le\max\{y,5E/y+6y^2\}
 \le3y+5E/y.
\]
The argument in Step 1 applies to every $\delta\ge0$, so no smallness condition on $E$ is needed.

\smallskip\noindent
\emph{Step 4: Translation.}
For general $a$, let $\widetilde\mu$ be the image of $\mu$ under $x\mapsto x-a$.
Then
\[
 G_{\widetilde\mu}(z)
 =\int_\R\frac{\mu(dx)}{z-(x-a)}
 =G_\mu(z+a),
 \qquad F_{\widetilde\mu}(z)=F_\mu(z+a).
\]
Since $\widetilde\mu((-\infty,x])=\mu((-\infty,x+a])$, we have $L(\mu,\widetilde\mu)\le|a|$.

The case $a=0$, applied to $\widetilde\mu$, and the triangle inequality now give
\begin{align*}
 L(\mu,\bbern)
 &\le L(\mu,\widetilde\mu)+L(\widetilde\mu,\bbern)\\
 &\le |a|+3y+
 \frac5y\max_{\sigma=\pm1}|F_\mu(a+\sigma+iy)-2iy|.
\end{align*}
This proves \eqref{eq:two}.
\end{proof}


\section{Upper bounds under moment assumptions}\label{sec:upper}
\begin{theorem}\label{thm:r-upper}
Let $r>2$, and suppose that $m_1(\mu)=0$, $m_2(\mu)=1$ and
\mbox{$\beta_r(\mu)<\infty$}. If $2<r\le5$, then
\begin{equation}\label{eq:r-upper}
 L(\mu_n,\bbern)\le C\beta_r(\mu)^{1/3}n^{-(r-2)/6}
 \qquad(n\ge1).
\end{equation}
The constant $C$ does not depend on $\mu,n,r$.
For $r\ge5$,
\begin{equation}\label{eq:r-large}
 L(\mu_n,\bbern)\le C\beta_r(\mu)^{1/(r-2)}n^{-1/2}.
\end{equation}
Moreover, for each fixed $\mu$, if $2<r<5$, then
\begin{equation}\label{eq:r-littleo}
 L(\mu_n,\bbern)=o\bigl(n^{-(r-2)/6}\bigr).
\end{equation}
\end{theorem}

\begin{proof}
In the first part of the proof we use only $m_1(\mu)=0$ and $m_2(\mu)=1$.
\[
 A_n=\frac{1+\displaystyle\int_{|x|\le\sqrt n/4}x^4\,\mu(dx)}n,
 \qquad B_n=\int_{|x|>\sqrt n/4}x^2\,\mu(dx).
\]
These quantities are finite. Indeed, $B_n\le1$, and
$x^4\le(n/16)x^2$ on $\{|x|\le\sqrt n/4\}$.
We prove the intermediate bound
\begin{equation}\label{eq:tail-bound}
 L(\mu_n,\bbern)\le C\bigl(A_n^{1/2}+B_n^{1/3}\bigr).
\end{equation}
Let $y=A_n^{1/2}+B_n^{1/3}$. Since $A_n\ge n^{-1}$, we have $y>0$.
If $y\ge1/8$, then $L\le1\le8y$ gives the required bound.
We may therefore assume that $0<y<1/8$. The definition of $y$ gives
\[
 n^{-1}\le A_n\le y^2,\qquad B_n\le y^3.
\]

Let
\[
 a=\frac1{2\sqrt n}\int_{|x|\le\sqrt n/4}x^3\,\mu(dx),
 \qquad z=\sigma+a+iy\quad(\sigma=\pm1).
\]
Only a truncated third moment is used here. By the Cauchy--Schwarz inequality,
\begin{align*}
 |a|^2
 &\le\frac1{4n}
   \left(\int_{|x|\le\sqrt n/4}x^2\,\mu(dx)\right)
   \left(\int_{|x|\le\sqrt n/4}x^4\,\mu(dx)\right)\\
 &\le\frac1{4n}\int_{|x|\le\sqrt n/4}x^4\,\mu(dx)
 \le\frac{A_n}{4}\le\frac{y^2}{4}.
\end{align*}
Thus $|a|\le y/2<1/16$. In particular,
\[
 |\operatorname{Re}z|=|\sigma+a|\ge1-|a|>\frac{15}{16}>\frac34,
 \qquad
 |z|\le1+|a|+y\le1+\frac32y<\frac{19}{16}<\frac32.
\]
Since $|z|\ge|\operatorname{Re}z|$, these bounds imply
\[
 |a|\le y/2,\qquad \frac12\le|z|\le\frac32,
 \qquad |\operatorname{Re}z|\ge\frac34.
\]
We will also use
\[
 |z-\sigma|\le|a|+y\le\frac32y,
 \qquad |z+\sigma|\le|z|+1\le\frac52.
\]

Define the weighted transform
\[
 S_n(z)=\int_\R\frac{x^2}{z-x/\sqrt n}\,\mu(dx).
\]
This integral converges absolutely. For real $x$,
$|z-x/\sqrt n|\ge\operatorname{Im}z=y$, so
\[
 \int_\R\left|\frac{x^2}{z-x/\sqrt n}\right|\,\mu(dx)
 \le\frac1y\int_\R x^2\,\mu(dx)=\frac1y.
\]
To relate this transform to $F_{\mu_n}$, put $w=\sqrt n z$ and use the identity
\[
 \frac{w}{w-x}=1+\frac{x}{w}+\frac{x^2}{w(w-x)}.
\]
All three terms on the right are integrable. The linear term is integrable
because $\int|x|\,\mu(dx)\le m_2(\mu)^{1/2}=1$, and the last term is
controlled by the preceding bound for $S_n(z)$.
Integrating and using $m_1(\mu)=0$, we obtain
\begin{align*}
 \sqrt n\,zG_\mu(\sqrt n z)
 &=1+\frac1{\sqrt n z}\int_\R x\,\mu(dx)
   +\int_\R\frac{x^2}{\sqrt n z(\sqrt n z-x)}\,\mu(dx)\\
 &=1+\frac1{nz}\int_\R\frac{x^2}{z-x/\sqrt n}\,\mu(dx)
 =1+\frac{S_n(z)}{nz}.
\end{align*}
The left side is nonzero, since $z\ne0$ and the imaginary part of
$G_\mu(\sqrt n z)$ is strictly negative.
Thus the denominator $1+S_n(z)/(nz)$ is nonzero.

Applying the transformation rules recalled in Section~\ref{sec:two},
with $c=1/\sqrt n$, and then using the identity above gives
\begin{equation}\label{eq:F-exact}
 \begin{aligned}
 F_{\mu_n}(z)
 &=\sqrt n F_\mu(\sqrt n z)-(n-1)z\\
 &=\frac{nz}{1+S_n(z)/(nz)}-(n-1)z
 =z-\frac{S_n(z)}{1+S_n(z)/(nz)}.
 \end{aligned}
\end{equation}
For the last equality, subtract $z$ and note that the remaining numerator is
$-nz\{S_n(z)/(nz)\}=-S_n(z)$.

We next estimate $S_n(z)$. Apply
\[
 \frac1{z-u}=\frac1z+\frac{u}{z^2}
                 +\frac{u^2}{z^2(z-u)}
\]
only on $\{|x|\le\sqrt n/4\}$, with $u=x/\sqrt n$.
The first two terms of the central integral are
\[
 \frac1z\int_{|x|\le\sqrt n/4}x^2\,\mu(dx)
 =\frac{1-B_n}{z},
 \qquad
 \frac1{\sqrt n z^2}\int_{|x|\le\sqrt n/4}x^3\,\mu(dx)
 =\frac{2a}{z^2}.
\]
Adding the tail integral and using
$B_n/z=\int_{|x|>\sqrt n/4}x^2/z\,\mu(dx)$, we obtain
\begin{align*}
 S_n(z)-\frac1z-\frac{2a}{z^2}
 ={}&\frac1{nz^2}\int_{|x|\le\sqrt n/4}
       \frac{x^4}{z-x/\sqrt n}\,\mu(dx)\\
 &+\int_{|x|>\sqrt n/4}x^2
       \left(\frac1{z-x/\sqrt n}-\frac1z\right)\mu(dx).
\end{align*}
In the central part,
\[
 |z-x/\sqrt n|\ge|\operatorname{Re}z|-|x|/\sqrt n
 \ge\frac34-\frac14=\frac12.
\]
Since $|z|^{-2}\le4$, the absolute value of its remainder is at most
\[
 \frac8n\int_{|x|\le\sqrt n/4}x^4\,\mu(dx)\le8A_n.
\]
In the tail, we use $|z-x/\sqrt n|\ge y$ and $|z|^{-1}\le2$.
Because $y<1/8<1/2$, we have
\[
 \left|\frac1{z-x/\sqrt n}-\frac1z\right|
 \le\frac1y+2\le\frac2y.
\]
The absolute value of the tail remainder is therefore at most $2B_n/y$.
Combining both bounds gives
\[
 \left|S_n(z)-\frac1z-\frac{2a}{z^2}\right|
 \le8A_n+\frac{2B_n}{y}\le8y^2+2y^2=10y^2.
\]

It follows that
\[
 |S_n(z)|\le\frac1{|z|}+\frac{2|a|}{|z|^2}+10y^2
 \le2+4y+10y^2
 <2+\frac12+\frac{10}{64}<3.
\]
Since $n^{-1}\le y^2$ and $|z|\ge1/2$, we also have
\[
 \left|\frac{S_n(z)}{nz}\right|\le\frac6n\le6y^2<\frac12,
 \qquad
 \left|1+\frac{S_n(z)}{nz}\right|
 \ge1-\left|\frac{S_n(z)}{nz}\right|>\frac12.
\]
Thus replacing the denominator in \eqref{eq:F-exact} by 1 gives an error bounded by
\begin{align*}
 \left|\frac{S_n(z)}{1+S_n(z)/(nz)}-S_n(z)\right|
 &=\frac{|S_n(z)|^2}
 {n|z|\,|1+S_n(z)/(nz)|}\\
 &\le\frac9{n(1/2)(1/2)}
 =\frac{36}{n}\le36y^2.
\end{align*}
Combining this with the remainder bound $10y^2$, we find
\[
 \left|F_{\mu_n}(z)-\left(z-\frac1z-\frac{2a}{z^2}\right)\right|
 \le46y^2.
\]
In particular,
\begin{equation}\label{eq:normalized-F}
 F_{\mu_n}(z)=z-\frac1z-\frac{2a}{z^2}+O(y^2),
\end{equation}
with an absolute constant in the error term.

We now use the choice of $a$ and $z$ to cancel the linear correction.
Since $\sigma^2=1$, direct expansion gives
\[
 z-\frac1z=2(z-\sigma)-\frac{(z-\sigma)^2}{z}.
\]
Using $2(z-\sigma)=2a+2iy$, we obtain the exact identity
\[
 z-\frac1z-\frac{2a}{z^2}-2iy
 =-\frac{(z-\sigma)^2}{z}
   +2a\left(1-\frac1{z^2}\right).
\]
The first term on the right has absolute value at most
\[
 \frac{|z-\sigma|^2}{|z|}
 \le2\left(\frac32y\right)^2=\frac92y^2.
\]
For the second term, factor $z^2-1=(z-\sigma)(z+\sigma)$ and use the bounds above:
\[
 2|a|\left|1-\frac1{z^2}\right|
 =\frac{2|a|\,|z-\sigma|\,|z+\sigma|}{|z|^2}
 \le y\cdot\frac32y\cdot\frac52\cdot4=15y^2.
\]
Hence the absolute value of the full expression is at most $39y^2/2$.
Together with the error $46y^2$, this shows that
\[
 |F_{\mu_n}(a+\sigma+iy)-2iy|
 \le\left(46+\frac{39}{2}\right)y^2\le66y^2
 \qquad(\sigma=\pm1).
\]
Lemma~\ref{lem:two} now gives
\[
 L(\mu_n,\bbern)
 \le|a|+3y+\frac5y\,66y^2
 \le\left(\frac12+3+330\right)y\le334y.
\]
This proves \eqref{eq:tail-bound} with $C=334$ (the case $y\ge1/8$ was treated at the beginning).

We now assume $2<r\le5$ and $\beta_r(\mu)<\infty$.
H\"older's inequality with conjugate exponents $r/2$ and $r/(r-2)$ gives
\[
 1=\int_\R|x|^2\,\mu(dx)
 \le\left(\int_\R|x|^r\,\mu(dx)\right)^{2/r}
       \left(\int_\R1\,\mu(dx)\right)^{(r-2)/r}
 =\beta_r(\mu)^{2/r}.
\]
Thus $\beta_r(\mu)\ge1$.

Let
\[
 h=\beta_r(\mu)^{1/3}n^{-(r-2)/6},
 \qquad h^3=\beta_r(\mu)n^{-(r-2)/2}.
\]
If $h\ge1$, the bound $L\le1\le h$ is enough.
We therefore consider $0<h<1$.
On $\{|x|>\sqrt n/4\}$, we have
\[
 |x|^2=|x|^r|x|^{-(r-2)}
 \le4^{r-2}n^{-(r-2)/2}|x|^r.
\]
Consequently,
\[
 B_n\le4^{r-2}\beta_r(\mu)n^{-(r-2)/2}
 =4^{r-2}h^3\le64h^3,
 \qquad B_n^{1/3}\le4h.
\]
The constant 64 is uniform in $r$, since $0<r-2\le3$.

For the central term, first suppose that $2<r\le4$.
On $\{|x|\le\sqrt n/4\}$, the nonnegative exponent $4-r$ gives
\[
 |x|^4\le(\sqrt n/4)^{4-r}|x|^r.
\]
It follows that
\[
 \frac1n\int_{|x|\le\sqrt n/4}x^4\,\mu(dx)
 \le4^{r-4}\beta_r(\mu)n^{-(r-2)/2}
 \le h^3.
\]
Here $4^{r-4}\le1$. Also, $0<(r-2)/2\le1$ and $n\ge1$, so
\[
 n^{-1}\le n^{-(r-2)/2}
 \le\beta_r(\mu)n^{-(r-2)/2}=h^3.
\]
Therefore
\[
 A_n\le2h^3,
 \qquad A_n^{1/2}\le\sqrt2\,h^{3/2}\le\sqrt2\,h.
\]
The last step uses $0<h<1$.

Next, suppose that $4<r\le5$.
The numbers $(r-4)/(r-2)$ and $2/(r-2)$ are positive and sum to 1, and
\[
 2\frac{r-4}{r-2}+r\frac2{r-2}=4.
\]
H\"older's inequality with conjugate exponents
$(r-2)/(r-4)$ and $(r-2)/2$ therefore gives
\begin{align*}
 m_4(\mu)
 &=\int_\R
   (|x|^2)^{(r-4)/(r-2)}(|x|^r)^{2/(r-2)}\,\mu(dx)\\
 &\le m_2(\mu)^{(r-4)/(r-2)}\beta_r(\mu)^{2/(r-2)}
 =\beta_r(\mu)^{2/(r-2)}.
\end{align*}
At $r=4$, the same bound is simply the identity $m_4(\mu)=\beta_4(\mu)$,
so no limiting form of H\"older's inequality is needed.
Thus, for all $4\le r\le5$, using $\beta_r(\mu)\ge1$, we obtain
\begin{align*}
 A_n^{1/2}
 &\le\left(\frac{1+m_4(\mu)}n\right)^{1/2}
 \le\sqrt2\,\beta_r(\mu)^{1/(r-2)}n^{-1/2}\\
 &=\sqrt2\,h^{3/(r-2)}\le\sqrt2\,h.
\end{align*}
The equality follows by raising the definition of $h$ to the power $3/(r-2)$.
The last inequality holds because $3/(r-2)\ge1$ and $0<h<1$.
At $r=5$, this exponent is exactly 1.
In both ranges of $r$, we have
\[
 A_n^{1/2}+B_n^{1/3}\le(\sqrt2+4)h.
\]
Combining this with \eqref{eq:tail-bound} proves \eqref{eq:r-upper}.
The constants remain independent of $r$.

For $r>5$, interpolate between the second and the $r$th absolute moments.
The numbers $(r-5)/(r-2)$ and $3/(r-2)$ are positive and sum to 1, and
\[
 2\frac{r-5}{r-2}+r\frac3{r-2}=5.
\]
H\"older's inequality with conjugate exponents
$(r-2)/(r-5)$ and $(r-2)/3$ gives
\begin{align*}
 \beta_5(\mu)
 &=\int_\R
  (|x|^2)^{(r-5)/(r-2)}(|x|^r)^{3/(r-2)}\,\mu(dx)\\
 &\le m_2(\mu)^{(r-5)/(r-2)}\beta_r(\mu)^{3/(r-2)}
 =\beta_r(\mu)^{3/(r-2)}.
\end{align*}
Applying the already proved bound at $r=5$, we obtain
\[
 L(\mu_n,\bbern)
 \le C\beta_5(\mu)^{1/3}n^{-1/2}
 \le C\beta_r(\mu)^{1/(r-2)}n^{-1/2}.
\]
This proves \eqref{eq:r-large} for $r>5$.
Its endpoint $r=5$ is already included in \eqref{eq:r-upper}.

Finally, fix $2<r<5$ and a measure $\mu$ satisfying the assumptions.
The tail estimate above, with $\int_{|x|>\sqrt n/4}|x|^r\,\mu(dx)$ in place of $\beta_r(\mu)$, gives
\[
 n^{(r-2)/2}B_n
 \le4^{r-2}\int_{|x|>\sqrt n/4}|x|^r\,\mu(dx)\longrightarrow0.
\]
The integral tends to zero because $|x|^r$ is integrable and the cutoff tends to infinity.
Taking cube roots gives
\[
 \frac{B_n^{1/3}}{n^{-(r-2)/6}}
 =\left(n^{(r-2)/2}B_n\right)^{1/3}\longrightarrow0.
\]
For $2<r\le4$, the central estimate above, which does not require $h<1$, gives
$A_n\le2\beta_r(\mu)n^{-(r-2)/2}$.
For $4\le r<5$, it gives $A_n\le(1+m_4(\mu))/n$.
Thus
\[
 A_n^{1/2}=
 \begin{cases}
 O_{\mu,r}(n^{-(r-2)/4}),&2<r\le4,\\
 O_{\mu,r}(n^{-1/2}),&4\le r<5.
 \end{cases}
\]
To compare these rates with the required one, divide by $n^{-(r-2)/6}$:
\[
 \frac{A_n^{1/2}}{n^{-(r-2)/6}}=
 \begin{cases}
 O_{\mu,r}(n^{-(r-2)/12}),&2<r\le4,\\
 O_{\mu,r}(n^{-(5-r)/6}),&4\le r<5.
 \end{cases}
\]
Both exponents are strictly negative in their stated ranges, so both ratios tend to zero.
The two estimates agree at $r=4$.
Therefore both $A_n^{1/2}$ and $B_n^{1/3}$ are $o(n^{-(r-2)/6})$.
The intermediate bound \eqref{eq:tail-bound} now proves \eqref{eq:r-littleo}.
\end{proof}


\section{Optimality of the convergence exponent}\label{sec:sharpness}

By Maassen's finite variance representation \cite[Proposition~2.2]{Maassen92},
for every probability measure $\rho$, there is a unique probability measure
$\Phi(\rho)$ with mean 0 and variance 1 such that
\begin{equation}\label{eq:Phi}
 F_{\Phi(\rho)}(z)=z-G_\rho(z).
\end{equation}
Let $K_\mu(z)=z-F_\mu(z)$ and $t=\sqrt n$.
The transformation rules in Section~\ref{sec:two} and \eqref{eq:Phi} give
\[
 K_{\mu_n}(z)=tK_\mu(tz)=tG_\rho(tz)
 =G_{D_{1/t}\rho}(z)\qquad\text{if }\mu=\Phi(\rho).
\]
Thus
\begin{equation}\label{eq:conjugacy}
 \mu=\Phi(\rho)\quad\Longrightarrow\quad
 \mu_n=\Phi(D_{1/\sqrt n}\rho).
\end{equation}

We construct $\rho$ from atoms that are far apart.

\begin{lemma}\label{lem:far}
Suppose that $b_1\ge8$, $b_{k+1}\ge4b_k$, $p_k>0$, and
$q=1-\sum_{k\ge1}p_k>0$. Let
\[
 \rho=q\delta_0+\sum_{k\ge1}p_k\delta_{b_k},\qquad \mu=\Phi(\rho).
\]
Then $\mu$ is purely atomic and has exactly one atom $\xi_k$ in each interval
$(b_k,b_{k+1})$. Its mass $w_k$ satisfies
\begin{equation}\label{eq:far}
 \frac{p_k}{2b_k}\le\xi_k-b_k\le\frac{2p_k}{b_k},\qquad
 \frac{p_k}{5b_k^2}\le w_k\le\frac{4p_k}{b_k^2}.
\end{equation}
The only other atoms of $\mu$ are one point in $(-1,0)$ and one point in $(0,2)$.
Thus, for every $u>0$,
\begin{equation}\label{eq:moment-series}
 \beta_u(\mu)<\infty
 \quad\Longleftrightarrow\quad
 \sum_{k\ge1}p_kb_k^{u-2}<\infty.
\end{equation}
\end{lemma}
\begin{proof}
By \eqref{eq:Phi},
\[
 F_\mu(z)=z-\frac qz-\sum_{j\ge1}\frac{p_j}{z-b_j}.
\]
The Cauchy transform is analytic off the support of its measure
and real on the complementary real intervals
\cite[Lemma~2.1]{Hasebe10}.
Here $\operatorname{supp}\rho=\{0,b_1,b_2,\ldots\}$ has no finite accumulation point.

Consequently, $F_\mu$ is meromorphic with simple poles at these points,
and differentiation under the integral gives, for real $x$ away from the poles,
\[
 F_\mu'(x)=1+\frac q{x^2}+\sum_{j\ge1}\frac{p_j}{(x-b_j)^2}>0.
\]
Thus $F_\mu$ is strictly increasing on every gap between its real poles.

For each gap $(b_k,b_{k+1})$, the limits of $F_\mu$ at the left and right
endpoints are $-\infty$ and $+\infty$, respectively.
The same signs hold at the endpoints of $(0,b_1)$.
On $(-\infty,0)$, we have $F_\mu(x)\to-\infty$ as $x\to-\infty$
and $F_\mu(x)\to+\infty$ as $x\uparrow0$.
For the first limit, one may use $|G_\rho(x)|\le1/|x|$ for $x<0$.
It follows that each of these intervals contains exactly one zero,
and each zero is simple because $F_\mu'>0$.
Moreover,
\[
 F_\mu(-1)=-1+q+\sum_j\frac{p_j}{1+b_j}
 =-\sum_j\frac{p_jb_j}{1+b_j}<0,
\]
and
\[
 F_\mu(2)=2-\frac q2+\sum_j\frac{p_j}{b_j-2}>0.
\]
Thus the first two zeros lie in $(-1,0)$ and $(0,2)$.

At each pole of $F_\mu$, its reciprocal $G_\mu$ has a removable singularity.
Away from the zeros of $F_\mu$, it therefore extends analytically across the real axis
with real values. By the support criterion in \cite[Lemma~2.1]{Hasebe10},
$\mu$ is supported on these zeros, which have no finite accumulation point.
At each simple zero $\xi$, the atom formula
\cite[Lemma~2.4]{Hasebe10} gives
\[
 \mu(\{\xi\})=\lim_{y\downarrow0}iyG_\mu(\xi+iy)
 =\frac1{F_\mu'(\xi)}>0.
\]
This proves that $\mu$ is purely atomic with exactly the stated atoms.

Fix $k$, and let $b=b_k$ and $p=p_k$ for the next estimates.
For $x\in[b,b+1]$, define
\[
 H_k(x)=x-\frac qx-\sum_{j\ne k}\frac{p_j}{x-b_j}.
\]
If $j<k$, then $b_j\le b/4$, so $x-b_j\ge3b/4$.
If $j>k$, then $b_j\ge4b$, so $b_j-x\ge3b-1$.
The terms with $j>k$ are positive in $H_k(x)$, whereas the terms with
$j<k$ and the term $-q/x$ are negative. Thus
\begin{align*}
 H_k(x)
 &\ge b-\frac qb-\frac4{3b}\sum_{j<k}p_j
 \ge b-\frac4{3b},\\
 H_k(x)
 &\le b+1+\frac1{3b-1}\sum_{j>k}p_j
 \le b+1+\frac1{3b-1}.
\end{align*}
All distances from $x$ to the poles in $H_k$ are at least $3b/4$.
Since $q+\sum_{j\ne k}p_j=1-p\le1$, we also have
\[
 1\le H_k'(x)
 =1+\frac q{x^2}+\sum_{j\ne k}\frac{p_j}{(x-b_j)^2}
 \le1+\frac{16}{9b^2}<2.
\]
Since $b\ge8$, these bounds imply $b/2<H_k(x)<2b$.

The two points $x_-=b+p/(2b)$ and $x_+=b+2p/b$ belong to $[b,b+1]$,
because $p<1$ and $b\ge8$. At these points,
\[
 F_\mu(x_-)=H_k(x_-)-2b<0,
 \qquad
 F_\mu(x_+)=H_k(x_+)-\frac b2>0.
\]
Therefore $x_-<\xi_k<x_+$.
Writing $h=\xi_k-b$, we obtain
\[
 \frac{p}{2b}\le h\le\frac{2p}{b},\qquad
 F_\mu'(\xi_k)=H_k'(\xi_k)+\frac p{h^2}.
\]
It follows that
\[
 \frac{b^2}{4p}\le F_\mu'(\xi_k)
 \le2+\frac{4b^2}{p}\le\frac{5b^2}{p}.
\]
The last inequality uses $2\le b^2/p$.
Taking reciprocals proves the mass bounds in \eqref{eq:far}.
Also, $b_k\le\xi_k\le2b_k$. Hence, for every $u>0$,
\[
 \frac15p_kb_k^{u-2}
 \le w_k\xi_k^u
 \le2^{u+2}p_kb_k^{u-2}.
\]
The two remaining atoms make a finite contribution to every absolute moment.
Summing these inequalities proves \eqref{eq:moment-series}.
\end{proof}

\begin{lemma}\label{lem:detune}
Under the assumptions of Lemma~\ref{lem:far}, let $n_k$ be positive integers,
$t_k=\sqrt{n_k}$, and $d_k=p_k^{1/3}$.
Suppose that $d_k\to0$ and that, for all sufficiently large $k$,
\[
 \frac{b_j}{t_k}\le\frac14\quad(j<k),\qquad
 \frac{b_j}{t_k}\ge4\quad(j>k).
\]
Let
\begin{equation}\label{eq:background-size}
 E_k=\frac1{t_k}\sum_{j<k}p_jb_j+\sum_{j>k}p_j.
\end{equation}
Then
\begin{equation}\label{eq:criterion}
 \frac{b_k}{t_k}=1+d_k+o(d_k),\qquad E_k=o(d_k)
 \quad\Longrightarrow\quad
 L(\mu_{n_k},\bbern)\ge\frac{d_k}{8}
 \quad(k\text{ sufficiently large}).
\end{equation}
\end{lemma}
\begin{proof}
Write $a=b_k/t_k$, $d=d_k$, and $c_j=b_j/t_k$.
By \eqref{eq:conjugacy} and $p_k=d^3$,
\[
 F_{\mu_{n_k}}(z)=H(z)-\frac{d^3}{z-a},\qquad
 H(z)=z-\frac qz-\sum_{j\ne k}\frac{p_j}{z-c_j}.
\]
The measure $q\delta_0+\sum_{j\ne k}p_j\delta_{c_j}$ has support disjoint
from $[1,3/2]$. The same analyticity property used in Lemma~\ref{lem:far}
shows that $H$ is analytic in a complex neighborhood of this interval
and real on the interval.
Using $q+\sum_{j\ne k}p_j=1-p_k$, we have
\[
 H(x)-\left(x-\frac{1-p_k}{x}\right)
 =-\sum_{j\ne k}p_j\left(\frac1{x-c_j}-\frac1x\right).
\]
For $1\le x\le3/2$ and $j<k$, we have $0<c_j\le1/4$, and therefore
\[
 \left|\frac1{x-c_j}-\frac1x\right|
 =\frac{c_j}{x(x-c_j)}\le\frac43c_j.
\]
For $j>k$, we have $c_j\ge4$, so
\[
 \left|\frac1{x-c_j}-\frac1x\right|
 =\frac1{c_j-x}+\frac1x\le\frac25+1<2.
\]
It follows that the absolute value of the preceding sum is at most
\[
 \frac4{3t_k}\sum_{j<k}p_jb_j+2\sum_{j>k}p_j\le2E_k.
\]
The distances from $x$ to all poles of $H$ are at least $3/4$.
Since all coefficients are nonnegative and have total mass $1-p_k\le1$,
\[
 1\le H'(x)=1+\frac q{x^2}+\sum_{j\ne k}\frac{p_j}{(x-c_j)^2}
 \le1+\frac{16}{9}.
\]
We have proved
\begin{equation}\label{eq:background}
 \left|H(x)-\left(x-\frac{1-p_k}{x}\right)\right|\le2E_k,
 \qquad 1\le H'(x)\le1+\frac{16}{9}<3.
\end{equation}

Let $s=a-1$. The assumptions give $s=d+o(d)$ and $E_k=o(d)$.
In particular, $a\to1$ and $a\in[1,3/2]$ for all sufficiently large $k$.
Since $p_k=d^3$, the exact identity
\[
 a-\frac{1-p_k}{a}
 =a-\frac1a+\frac{d^3}{a}
 =2s-\frac{s^2}{a}+\frac{d^3}{a}
\]
shows that
\[
 H(a)=2s-\frac{s^2}{a}+\frac{d^3}{a}+O(E_k)=2d+o(d).
\]
Thus, for all sufficiently large $k$,
\[
 d\le\frac18,\qquad a-1\ge\frac d2,\qquad
 [a,a+d^2]\subset[1,3/2],\qquad
 \frac74d\le H(a)\le\frac94d.
\]
The derivative bounds in \eqref{eq:background} now give
\begin{align*}
 H\left(a+\frac25d^2\right)
 &\le H(a)+3\cdot\frac25d^2
 \le\frac94d+\frac65d^2<\frac52d,\\
 H\left(a+\frac23d^2\right)
 &\ge H(a)\ge\frac74d>\frac32d.
\end{align*}
Here the strict upper bound follows from $d\le1/8$, since
$9/4+(6/5)d\le12/5<5/2$.
At the two points above, the pole term $d^3/(x-a)$ equals $5d/2$
and $3d/2$, respectively. Hence
\[
 F_{\mu_{n_k}}\left(a+\frac25d^2\right)<0,
 \qquad
 F_{\mu_{n_k}}\left(a+\frac23d^2\right)>0.
\]
Also,
\[
 F_{\mu_{n_k}}'(x)=H'(x)+\frac{d^3}{(x-a)^2}>0
 \qquad(a<x\le a+d^2).
\]
It follows that there is exactly one zero $\lambda_k$ between these two points,
and that this zero is simple. The residue calculation from Lemma~\ref{lem:far}
therefore gives an atom of $\mu_{n_k}$ at $\lambda_k$, with mass
\begin{align*}
 w_k:=\mu_{n_k}(\{\lambda_k\})
 &=\left(H'(\lambda_k)+\frac{d^3}{(\lambda_k-a)^2}\right)^{-1}\\
 &\ge\left(3+\frac{25}{4d}\right)^{-1}
 =\frac{d}{3d+25/4}\ge\frac d8.
\end{align*}
Indeed, $\lambda_k-a>2d^2/5$ and $3d+25/4\le53/8<8$.
Here $F_{\mu_{n_k}}$ is analytic in a neighborhood of
$\lambda_k$ and has a simple zero there.
Hence $G_{\mu_{n_k}}=1/F_{\mu_{n_k}}$ has a simple pole
at $\lambda_k$, with residue
$1/F'_{\mu_{n_k}}(\lambda_k)$.
Also, $\lambda_k-1\ge a-1\ge d/2$.

For completeness, let $\mu$ be any probability measure with an atom of
mass $w$ at $v>1$.
If $0<\varepsilon<\min\{v-1,w\}$, then $1+\varepsilon<v$, and hence
\[
 \mu((-\infty,1+\varepsilon])\le1-w
 <1-\varepsilon=\bbern((-\infty,1])-\varepsilon.
\]
This violates the defining inequality for L\'evy distance at $x=1$.
Thus $L(\mu,\bbern)\ge\min\{v-1,w\}$.
Taking $v=\lambda_k$ gives
\[
 L(\mu_{n_k},\bbern)
 \ge\min\{\lambda_k-1,w_k\}
 \ge\min\{d/2,d/8\}=d/8,
\]
as required.
\end{proof}

\begin{theorem}[Optimality of the exponent for a fixed probability measure]\label{thm:family}
For a real number $r\ge2$, define
\begin{equation}\label{eq:family}
 \begin{gathered}
 T_k=8^k,\qquad p_k=\frac1{64k^2T_k^{r-2}},\qquad
 d_k=p_k^{1/3},\qquad b_k=T_k(1+d_k),\\
 \rho=\left(1-\sum_{k\ge1}p_k\right)\delta_0+
       \sum_{k\ge1}p_k\delta_{b_k},\qquad \mu=\Phi(\rho).
 \end{gathered}
\end{equation}
Then $\mu$ has mean 0 and variance 1, and for every $u>0$,
\begin{equation}\label{eq:exact-threshold}
 \beta_u(\mu)<\infty\quad\Longleftrightarrow\quad 0<u\le r.
\end{equation}
If $2\le r<5$, then along $n_k=T_k^2=64^k$, for all sufficiently large $k$,
\begin{equation}\label{eq:family-lower}
 L(\mu_{n_k},\bbern)\ge\frac{d_k}{8}
 =\frac1{32k^{2/3}}n_k^{-(r-2)/6}.
\end{equation}
Thus, if $2\le r<5$, this single probability measure satisfies
\begin{equation}\label{eq:sharp-exponent}
 \sup_{n\ge1}n^\gamma L(\mu_n,\bbern)=\infty
 \qquad\left(\gamma>\frac{r-2}{6}\right).
\end{equation}
\end{theorem}
\begin{proof}
Let $s=r-2\ge0$. Then
\[
 d_k=\frac1{4k^{2/3}T_k^{s/3}}\le\frac14,
 \qquad
 \sum_{k\ge1}p_k\le\frac1{64}\sum_{k\ge1}\frac1{k^2}
 =\frac{\pi^2}{384}<1.
\]
Also, $b_1\ge T_1=8$, and
\[
 \frac{b_{k+1}}{b_k}
 =8\frac{1+d_{k+1}}{1+d_k}
 \ge\frac8{1+d_k}\ge\frac{32}{5}>4.
\]
Thus Lemma~\ref{lem:far} applies.
For $u>0$, the terms of the series in \eqref{eq:moment-series} are
\[
 p_kb_k^{u-2}
 =\frac{(1+d_k)^{u-2}}{64k^2}T_k^{u-r}.
\]
Since $1\le1+d_k\le5/4$, the factor $(1+d_k)^{u-2}$ is bounded above
and below by positive constants depending only on $u$.
This remains true when $0<u<2$.
Consequently,
\[
 \sum_kp_kb_k^{u-2}\asymp_{r,u}
 \sum_k\frac{8^{(u-r)k}}{k^2}.
\]
If $u\le r$, this series is bounded by $\sum_k k^{-2}$.
If $u>r$, its terms do not tend to zero, since an exponential factor
with base greater than 1 grows faster than $k^2$.
Hence \eqref{eq:moment-series} gives \eqref{eq:exact-threshold}.

Now suppose that $2\le r<5$.
For $n_k=T_k^2$ we have $t_k=T_k$ and $b_k/t_k=1+d_k$.
For $j<k$ and $j>k$, respectively, we have
\[
 \frac{b_j}{T_k}\le\frac54\,8^{j-k}\le\frac5{32}<\frac14,
 \qquad
 \frac{b_j}{T_k}\ge8^{j-k}\ge8>4.
\]
Also, $d_k\to0$. Thus it remains to verify $E_k=o(d_k)$ in order to apply
Lemma~\ref{lem:detune}.

First let $2<r<5$, so $0<s<3$, and put $\theta=\min\{s,1\}$.
If $0<s\le1$, then $\theta=s>s/3$.
If $1<s<3$, then $\theta=1>s/3$.
In both cases, $s/3<\theta\le1$ and $\theta\le s$.
Moreover,
\[
 p_jb_j^\theta
 =\frac{(1+d_j)^\theta}{64j^2}T_j^{\theta-s}
 \le\frac{(5/4)^\theta}{64j^2}.
\]
Hence $M_\theta:=\sum_jp_jb_j^\theta<\infty$.
For $0<c<1$ and $0<\theta\le1$, we have $c\le c^\theta$.
For $c>1$ and $\theta>0$, we have $1\le c^\theta$.
Applying these inequalities with $c=b_j/T_k$ to the two sums in
\eqref{eq:background-size} gives
\begin{align*}
 E_k
 &\le\sum_{j<k}p_j\left(\frac{b_j}{T_k}\right)^\theta
     +\sum_{j>k}p_j\left(\frac{b_j}{T_k}\right)^\theta\\
 &\le T_k^{-\theta}M_\theta.
\end{align*}
Using the formula for $d_k$, we obtain
\[
 \frac{E_k}{d_k}
 \le4M_\theta k^{2/3}T_k^{-\theta+s/3}
 =4M_\theta k^{2/3}8^{-k(\theta-s/3)}\longrightarrow0.
\]
Indeed, $\theta-s/3>0$, and the logarithm of the last varying factor is
$(2/3)\log k-k(\theta-s/3)\log8$, which tends to $-\infty$.

For $r=2$, we have $p_j=1/(64j^2)$ and $d_k=1/(4k^{2/3})$.
If $j\le k/2$, then $b_j/T_k\le(5/4)8^{-k/2}$.
If $k/2<j<k$, then $b_j/T_k\le5/32<1$.
It follows that
\begin{align*}
 E_k
 &\le\frac54\,8^{-k/2}\sum_{j\le k/2}p_j
      +\sum_{k/2<j<k}p_j+\sum_{j>k}p_j\\
 &\le\frac54\,8^{-k/2}\sum_jp_j+\sum_{j>k/2}p_j.
\end{align*}
For $k\ge2$, the integral bound for the last sum gives
\[
 \sum_{j>k/2}p_j
 =\frac1{64}\sum_{j>\lfloor k/2\rfloor}\frac1{j^2}
 \le\frac1{64}\int_{\lfloor k/2\rfloor}^{\infty}\frac{dx}{x^2}
 =\frac1{64\lfloor k/2\rfloor}=O(k^{-1}).
\]
Thus $E_k=O(k^{-1})$ and $E_k/d_k=O(k^{-1/3})\to0$.
In both cases, Lemma~\ref{lem:detune} gives \eqref{eq:family-lower}.
Finally, if $\gamma>(r-2)/6$, then
\[
 n_k^\gamma L(\mu_{n_k},\bbern)
 \ge\frac1{32k^{2/3}}64^{k(\gamma-(r-2)/6)}\longrightarrow\infty.
\]
This proves \eqref{eq:sharp-exponent} for the single distribution already
defined in \eqref{eq:family}.
\end{proof}

The factor $k^{-2/3}$ remains in the lower bound in Theorem~\ref{thm:family}.
Thus this bound is consistent with the little-$o$ estimate in
Theorem~\ref{thm:r-upper} for $2<r<5$.
Optimality means that, in general, the exponent cannot be increased by any positive amount.

\begin{proposition}\label{prop:drift}
Suppose that $\rho$ is supported on $[0,\infty)$,
$\rho(\{0\})>0$, and $0<\int x\,d\rho(x)<\infty$.
For $\mu=\Phi(\rho)$,
\begin{equation}\label{eq:drift}
 \liminf_{n\to\infty}\sqrt n\,L(\mu_n,\bbern)
 \ge\frac12\int x\,d\rho(x)>0.
\end{equation}
In particular, for the example in \eqref{eq:family} with $r\ge5$,
\[
 0<\liminf_n\sqrt n\,L(\mu_n,\bbern)
 \le\limsup_n\sqrt n\,L(\mu_n,\bbern)<\infty.
\]
\end{proposition}

\begin{proof}
Let $t=\sqrt n$ and $q_0=\rho(\{0\})>0$.
By \eqref{eq:conjugacy}, for $x<0$,
\[
 F_{\mu_n}(x)=x-\int\frac1{x-v/t}\,d\rho(v),\qquad
 F_{\mu_n}'(x)=1+\int\frac1{(x-v/t)^2}\,d\rho(v)>0.
\]
These formulas and differentiation are valid because the support of $\rho$
is nonnegative and, on each compact subinterval of $(-\infty,0)$,
the denominators are bounded away from zero uniformly in $v$.
Since $\int v\,d\rho(v)>0$, the set $(0,\infty)$ has positive $\rho$-mass.
Therefore
\[
 F_{\mu_n}(-1)=-1+\int\frac1{1+v/t}\,d\rho(v)<0.
\]
Also, $F_{\mu_n}(x)\ge x-q_0/x\to+\infty$ as $x\uparrow0$.
Strict monotonicity now gives exactly one negative zero $-u_t$, with
$0<u_t<1$.
On the negative real axis away from this zero, $G_{\mu_n}=1/F_{\mu_n}$
extends analytically with real values. The support criterion
\cite[Lemma~2.1]{Hasebe10} therefore excludes any point of
$\operatorname{supp}\mu_n$ to the left of $-u_t$.

The zero equation $F_{\mu_n}(-u_t)=0$ reads
\[
 u_t=\int\frac1{u_t+x/t}\,d\rho(x).
\]
The atom at zero contributes $q_0/u_t$ to this integral.
Thus $u_t\ge q_0/u_t$, or $u_t^2\ge q_0$.
Multiplying the zero equation by $u_t$ and subtracting from
$1=\int1\,d\rho$ gives
\[
 1-u_t^2=\int\frac{x/t}{u_t+x/t}\,d\rho(x).
\]
The integrand lies between 0 and 1.
For each fixed $x\ge0$, it tends to zero because $u_t\ge\sqrt{q_0}$.
The dominated convergence theorem therefore gives $u_t^2\to1$,
and hence $u_t\to1$.
Multiplying the last identity by $t$, we obtain
\[
 t(1-u_t^2)=\int\frac{x}{u_t+x/t}\,d\rho(x)
 \longrightarrow\int x\,d\rho(x).
\]
For this application of dominated convergence, the pointwise limit of the
integrand is $x$, and
\[
 0\le\frac{x}{u_t+x/t}\le\frac{x}{\sqrt{q_0}}.
\]
The right-hand side is integrable by the assumed finite first moment of $\rho$.
It follows that
\[
 t(1-u_t)=\frac{t(1-u_t^2)}{1+u_t}
 \longrightarrow\frac12\int x\,d\rho(x).
\]

If $0<\varepsilon<\min\{1-u_t,1/2\}$, then $-1+\varepsilon<-u_t$.
Since $\mu_n$ has no mass to the left of $-u_t$,
\[
 \mu_n((-\infty,-1+\varepsilon])=0
 <\frac12-\varepsilon=\bbern((-\infty,-1])-\varepsilon.
\]
The defining inequality for L\'evy distance therefore fails at $x=-1$,
so $L(\mu_n,\bbern)\ge\min\{1-u_t,1/2\}$.
Since $u_t\to1$, the minimum equals $1-u_t$ for all sufficiently large $t$.
Multiplying by $t$ and taking the lower limit proves \eqref{eq:drift}.

For the family in \eqref{eq:family} with $r\ge5$,
\[
 0<\int x\,d\rho(x)=\sum_kp_kb_k
 =\frac1{64}\sum_k\frac{1+d_k}{k^2}T_k^{3-r}
 \le\frac5{256}\sum_k\frac{T_k^{3-r}}{k^2}<\infty.
\]
Also, $\rho(\{0\})>0$ and $\beta_r(\mu)<\infty$ by Theorem~\ref{thm:family}.
Thus \eqref{eq:drift} supplies the positive lower limit, and
Theorem~\ref{thm:r-upper} supplies the finite upper limit.
\end{proof}

The exponent $1/2$ cannot be improved even for bounded support.
Indeed, Arizmendi--Salazar \cite[Example~1]{AS18} exhibit the centered,
variance-one two-point distribution $\mu=\Phi(\delta_1)$ with
\[
 F_\mu(z)=\frac{z^2-z-1}{z-1},\qquad
 L(\mu_n,\bbern)\ge\frac1{6\sqrt n}\quad(n\ge1).
\]
All absolute moments of this distribution are finite, so the exponent $1/2$ cannot be improved even for distributions with compact support.


\section{The case of only a finite second moment}\label{sec:variance}
\begin{proposition}\label{prop:variance}
If $m_1(\mu)=0$ and $m_2(\mu)=1$, then $L(\mu_n,\bbern)\to0$.
\end{proposition}
\begin{proof}
This is the Boolean central limit theorem \cite[Theorem~3.4]{SW97}.
For completeness, we record its short deduction from \eqref{eq:tail-bound},
whose derivation used only the second moment.
The two integrands below converge pointwise to zero, and
\[
 0\le x^2\mathbf1_{\{|x|>\sqrt n/4\}}\le x^2,\qquad
 0\le\frac{x^4}{n}\mathbf1_{\{|x|\le\sqrt n/4\}}\le\frac{x^2}{16}.
\]
Dominated convergence gives $B_n\to0$ and $A_n\to0$.
Hence \eqref{eq:tail-bound} yields
$L(\mu_n,\bbern)\le C(A_n^{1/2}+B_n^{1/3})\to0$.
\end{proof}

The example with $r=2$ in Theorem~\ref{thm:family} has infinite absolute moments
of every order greater than two. For all sufficiently large $k$, it satisfies
\[
 L(\mu_{64^k},\bbern)\ge\frac1{32k^{2/3}}.
\]
For every $\gamma>0$, it follows that
\[
 (64^k)^\gamma L(\mu_{64^k},\bbern)
 \ge\frac{64^{\gamma k}}{32k^{2/3}}\longrightarrow\infty,
\]
because the logarithm of the right-hand side is
$\gamma k\log64-(2/3)\log k-\log32$, which tends to infinity.
Thus no positive polynomial rate can be guaranteed from the second moment alone.
By choosing the atom locations while keeping their masses fixed, we obtain the following result.

\begin{theorem}[Arbitrarily slow convergence]\label{thm:arbitrary-slow}
For any positive sequence $\varepsilon_n\to0$, there is a single probability measure
$\mu$ with mean 0 and variance 1 such that $\beta_u(\mu)=\infty$ for every $u>2$ and
\begin{equation}\label{eq:arbitrary-slow}
 L(\mu_n,\bbern)\longrightarrow0,
 \qquad
 \limsup_{n\to\infty}\frac{L(\mu_n,\bbern)}{\varepsilon_n}=\infty.
\end{equation}
\end{theorem}
\begin{proof}
\emph{Step 1: Choice of the atom locations.}
Set $p_k=1/(64k^2)$ and $d_k=p_k^{1/3}=1/(4k^{2/3})$.
We choose positive integers $T_k$ inductively, and then put $b_k=T_k(1+d_k)$.
At each step, we require
\begin{equation}\label{eq:recursive}
 \begin{gathered}
 T_1\ge8,\qquad T_k\ge8T_{k-1}\quad(k\ge2),\\
 \frac1{T_k}\sum_{j<k}p_jb_j\le\frac{d_k}{k},\qquad
 \varepsilon_{T_k^2}\le\frac{d_k}{k^2}.
 \end{gathered}
\end{equation}
To justify this choice, suppose that $b_1,\ldots,b_{k-1}$ have already been chosen,
and write $S_k=\sum_{j<k}p_jb_j$. This is a finite number; for $k=1$, it is zero.
Since $\varepsilon_n\to0$ and $d_k/k^2>0$, there exists $N_k$ such that
\[
 \varepsilon_n\le d_k/k^2\qquad(n\ge N_k).
\]
Choose an integer $T_k$ which is at least $8$, $kS_k/d_k$, and $\sqrt{N_k}$,
and, if $k\ge2$, at least $8T_{k-1}$.
Then $S_k/T_k\le d_k/k$ and $T_k^2\ge N_k$, so all conditions in
\eqref{eq:recursive} hold. This uses convergence of $\varepsilon_n$, without a monotonicity assumption.

\smallskip\noindent
\emph{Step 2: Definition of a single distribution and its moments.}
After choosing the whole sequence, define
\[
 \rho=\left(1-\sum_{k\ge1}p_k\right)\delta_0
       +\sum_{k\ge1}p_k\delta_{b_k},\qquad \mu=\Phi(\rho).
\]
Here $\sum p_k=\pi^2/384<1$, so $\rho$ is a probability measure with positive mass at zero.
Moreover, $d_k\le1/4$ and
\[
 T_k\ge8^k,\qquad b_1\ge8,\qquad
 \frac{b_{k+1}}{b_k}
 =\frac{T_{k+1}}{T_k}\frac{1+d_{k+1}}{1+d_k}
 \ge\frac8{5/4}=\frac{32}{5}>4.
\]
Thus Lemma~\ref{lem:far} applies. For every $u>2$,
\[
 \sum_kp_kb_k^{u-2}
 \ge\frac1{64}\sum_k\frac{8^{(u-2)k}}{k^2}=\infty.
\]
Indeed, the summands on the right do not tend to zero:
their logarithms equal $(u-2)k\log8-2\log k$ and tend to infinity.
It follows from \eqref{eq:moment-series} that $\beta_u(\mu)=\infty$ for all $u>2$.
The mean and variance of $\mu$ are 0 and 1 by the definition of $\Phi$.
This defines one measure $\mu$ for the given sequence $\varepsilon_n$.

\smallskip\noindent
\emph{Step 3: The lower bound along the chosen times.}
Set $n_k=T_k^2$, so that $t_k=\sqrt{n_k}=T_k$.
The $k$th atom of $\rho$ is mapped to
$b_k/T_k=1+d_k$.
For $j<k$, repeated use of $T_{\ell+1}\ge8T_\ell$ gives
\[
 \frac{b_j}{T_k}
 \le\frac54\frac{T_j}{T_k}
 \le\frac54\,8^{j-k}\le\frac5{32}<\frac14.
\]
For $j>k$, the same growth condition gives
\[
 \frac{b_j}{T_k}\ge\frac{T_j}{T_k}\ge8^{j-k}\ge8>4.
\]
Thus the separation conditions in Lemma~\ref{lem:detune} hold.
The last condition to check is $E_k=o(d_k)$.
Since $x\mapsto x^{-2}$ is decreasing,
\[
 \sum_{j>k}p_j
 =\frac1{64}\sum_{j=k+1}^\infty\frac1{j^2}
 \le\frac1{64}\int_k^\infty\frac{dx}{x^2}
 =\frac1{64k}.
\]
Therefore \eqref{eq:background-size} and \eqref{eq:recursive} imply
\[
 E_k\le\frac{d_k}{k}+\frac1{64k},
 \qquad
 \frac{E_k}{d_k}\le\frac1k+\frac1{16k^{1/3}}\longrightarrow0.
\]
Since also $d_k\to0$, Lemma~\ref{lem:detune} yields
$L(\mu_{n_k},\bbern)\ge d_k/8$ for all sufficiently large $k$.
Using the last condition in \eqref{eq:recursive}, we obtain
\[
 \frac{L(\mu_{n_k},\bbern)}{\varepsilon_{n_k}}
 \ge\frac{d_k/8}{d_k/k^2}
 =\frac{k^2}{8}\longrightarrow\infty.
\]
The sequence $n_k$ tends to infinity, so this proves the required infinite limit superior.
Finally, Proposition~\ref{prop:variance} applies to this same fixed $\mu$ and gives
$L(\mu_n,\bbern)\to0$ along the full sequence.
\end{proof}

In this theorem, a single $\mu$ is fixed once the prescribed rate sequence is given.
Thus, even if a finite constant $C(\mu)$ is allowed to depend on the distribution,
the bound $C(\mu)\varepsilon_n$ cannot be guaranteed for all distributions with mean 0 and variance 1.


\end{document}